\documentclass{article}
\usepackage{amsmath,amssymb,bm,mathbbol,bbm,amsthm}
\usepackage{latexsym}
\usepackage[dvipdfmx]{graphicx}
\usepackage{subfigure}
\usepackage{multirow}
\usepackage{ascmac}

\allowdisplaybreaks

\newtheorem{theorem}{Theorem}
\newtheorem{corollary}[theorem]{Corollary}
\newtheorem{lemma}[theorem]{Lemma}

\newtheorem{remark}[theorem]{Remark}

\newtheorem{assumption}[theorem]{Assumption}

\newcommand{\ds}{\displaystyle}
\newcommand{\E}{\mathrm{E}}

\newcommand{\M}{\mathcal{M}}
\newcommand{\K}{\mathcal{K}}
\newcommand{\one}{\mbox{$\mathbb 1$}}
\newcommand{\PR}{\mathrm{PR}}
\newcommand{\NP}{\mathrm{NP}}
\newcommand{\on}{\mathrm{on}}

\begin{document}

\begin{center}
{\Large\textbf{Mean Waiting Times in Discrete-Time Priority Queues }}

\medskip

{\Large\textbf{with Geometrically Distributed Idle Periods }}

\bigskip

\bigskip

Tetsuya Takine

\bigskip

Department of Information and Communications Technology, 

Graduate School of Engineering, The University of Osaka, Japan

\end{center}

\medskip

\bigskip

\noindent
\textbf{Abstract } 

\noindent
This paper considers the mean waiting times in discrete-time
preemptive-resume and nonpreemptive priority single-server queues fed
by $K$ independent batch Markovian arrival streams with geometrically
distributed idle periods. While being active, the $k$th
($k=1,2,\ldots,K$) arrival stream feeds at least one customer to the
queue, where the number of arriving customers depends on the state of
the underlying Markov chain. Service times of class $k$ customers are
independent and identically distributed according to a general
distribution. For these queues, we derive explicit formulae for the
mean waiting times of customers in each class.

\bigskip

\noindent
\textbf{Keywords} 

\noindent
discrete-time queue, batch Markovian arrival streams, conservation
law, preemptive-resume priority, nonpreemptive priority, mean waiting
time

\bigskip


\bigskip

\section{Introduction}\label{sec-introduction}

Priority queues have been studied extensively in the literature
\cite{Jais68,Mill60,Taka68,Taka91,Taka93}. See \cite{Walr04} and
references therein also.  When arrivals are memoryless (i.e., the
arrival process has independent and stationary increment properties),
explicit formulae for the mean waiting time can be obtained
\cite{Rubi89,Taka91,Taka93}. Here we use the word ``explicit'' to
imply that there is no need to solve a system of equations or a
fixed-point equation to evaluate the formula, which arise typically in
queueing analysis. When arrivals are correlated in time, however, we
cannot obtain such an explicit formula in general. See
\cite{Kham92,Taki94a,Taki94b,Taki96,Taki99,Walr05} for example.

This paper considers discrete-time priority single-server queues with
a superposition of $K$ independent batch Markovian arrival streams of
customers. Customers from the $k$th arrival stream are referred to as
class $k$ customers. Time is divided into slots of equal length.  Each
arrival stream becomes idle and active alternately, and idle and
active periods form a discrete-time alternating renewal process. The
lengths of idle periods are geometrically distributed and no customers
arrive during idle periods. On the other hand, the lengths of active
periods follow a discrete phase-type distribution, which is equivalent
to the distribution of first passage times to the absorbing state in
an underlying absorbing Markov chain, and at least one customer
arrives at each slot in active periods. Note that 
the number of customers arriving at each slot in active periods
can depend on the state of the underlying absorbing Markov
chain. Service times of customers in each class are independent and
identically distributed (i.i.d.) according to a general distribution.

The main purpose of this paper is the derivation of explicit formulae
for the mean waiting times in discrete-time preemptive-resume and
nonpreemptive priority single-server queues with the above-mentioned
arrival streams. The advantage of discrete-time queues over the
continuous-time counterparts is that even when arrivals are correlated
in time, we may be able to derive explicit formulae for performance
measures of interest. For example, Takine derived the mean 
amount of unfinished work 
in the discrete-time work-conserving single-server queue with
the above-mentioned arrival streams~\cite{Taki05}. See \cite{Neut90}
and \cite{Vite86} also. To the best of our knowledge, however, there
are no reports on an explicit formula for the mean waiting time in
$K$-class priority queues with arrivals correlated in time.

The essential assumption in the model is that (i) arrival streams are
independent of each other, (ii) idle periods in each stream follow a
geometric distribution, and (iii) each stream feeds at least one
customer at each slot in active periods. These assumptions yield an
i.i.d.\ sequence of busy cycle during which customers are served
successively. We first derive the conservation law for mean waiting
times in work-conserving, nonpreemptive-service single-server queues,
using the general formula for the mean amount of unfinished work in
\cite{Taki05}.  Note that the mean waiting times in the nonpreemptive
priority queue satisfy this conservation law.  Next, we consider the
preemptive-resume priority queue and derive an explicit formula for
the mean waiting time. We then consider the relation between mean
waiting times in nonpreemptive and preemptive-resume priority queues
and derive the mean waiting time in the nonpreemptive priority queue
in terms of that in the preemptive-resume priority queue.

The rest of this paper is organized as follows.  Section
\ref{section:model} describes the mathematical model. In
Section~\ref{section:law}, we first establish the conservation law in
work-conserving single-server queues with nonpreemptive services. We
then discuss the first two moments of busy cycle lengths.  In
Section~\ref{sec-preemptive}, we consider the preemptive-resume
priority queue and in Section~\ref{sec-nonpreemptive}, we consider the
nonpreemptive priority queue.  Finally, we provide the results for
some important special cases in Section \ref{section:special}.

Throughout this paper, we use the following conventions.
For any function $f(z)$ of $z$, let $f^{(n)}$ denote the $n$th
derivative of $f(z)$ evaluated at $z=1-$, i.e.,
\[
f^{(n)} = \lim_{z \to 1-} \frac{d^n}{dz^n} f(z).
\]
This convention applies to vector and matrix functions, too. Furthermore,
for any nonnegative integer random variable $F$, let $\tilde{F}$
denote an equilibrium random variable.
\[
\Pr (\tilde{F} = n ) = \frac{\Pr(F > n)}{\E[F]}, 
\quad 
n=0,1,\ldots,
\]
and 
\[
\E[ \tilde{F}] 
=
\frac{\E[F(F-1)]}{2 \E[F]}. 
\]
Note that the forward recurrence time of $F$ follows the equilibrium
distribution of $F$. Finally, the empty sum is defined as zero.

\section{Model}\label{section:model}

We consider discrete-time work-conserving single-server queues with
$K$ independent batch Markovian arrival streams of customers. Let $\K
= \{1,2,\ldots,K\}$. Customers arriving from the $k$th ($k \in \K$)
arrival stream are referred to as class $k$ customers. Let $A_{k,n}$
($k \in \K$, $n=1,2,\ldots$) denote the number of class $k$ customers
arriving at the $n$th slot. The sequence $\{A_{k,n}\}_{n=1,2,\ldots}$
is characterized as follows.  For each $k$ ($k \in \K$), we assume
that there exists an irreducible underlying Markov chain with finite
state space $\M_k =\{0, 1, \ldots, M_k\}$, which governs class $k$
arrivals. To avoid triviality, we assume $M_k \geq 1$. Let $S_{k,n}$
($k \in \K$, $n=0,1,\ldots$) denote the state of the underlying Markov
chain for class $k$ arrivals at the $n$th slot.
\begin{assumption}\label{assumption-1}
$A_{k,n}$ ($k \in \K$, $n=1,2,\ldots$) depends only on $S_{k,n-1}$ and
$S_{k,n}$.
\end{assumption}
\begin{assumption}\label{assumption-2}
The following relationships hold for all $k$ ($k \in \K$) and for all
$n$ ($n=1,2,\ldots$).
\[
S_{k,n} = 0 \, \Leftrightarrow\,  A_{k,n} = 0,
\qquad
S_{k,n} \in \M_k^{[\on]}\,  \Leftrightarrow\,  A_{k,n} \geq 1,
\]
where $\M_k^{[\on]} = \{1,2,\ldots M_k\}$ ($k \in \K$).
\end{assumption}
Accordingly, we assume that the irreducible transition probability
matrix $\bm{P}_k$ ($k \in \K$) of the underlying Markov chain for 
class $k$ is given by
\begin{equation}
\bm{P}_k = \left(\begin{array}{cc}
p_k & \left(1-p_k\right) \bm{\alpha}_k
\\
(\bm{I}_k-\bm{T}_k)\bm{e}_k & \bm{T}_k
\end{array}\right),
\quad
k \in \K,
\label{P_k}
\end{equation}
where $\bm{\alpha}_k$ and $\bm{T}_k$ denote a $1 \times M_k$
probability vector and an $M_k \times M_k$ substochastic matrix,
$\bm{I}_k$ denotes an $M_k \times
M_k$ unit matrix, and $\bm{e}_k$ denotes an $M_k \times 1$
vector whose elements are all equal to one. Because $\bm{P}_k$ is
assumed to be irreducible, $0 \leq p_k < 1$, $(\bm{I}_k-\bm{T}_k)
\bm{e}_k \neq \bf{0}$, and the stochastic matrix $\bm{T}_k +
(\bm{I}_k-\bm{T}_k)\bm{e}_k\bm{\alpha}_k$ is irreducible.

All arriving customers are accommodated in the buffer of infinite
capacity and they are served by a single server. Let $B_{k,n}$ ($k \in
\K$, $n=1,2,\ldots$) denote the amount of class $k$ work arriving at
the $n$th slot. We then have
\[
B_{k,n} = \sum_{\ell=1}^{A_{k,n}} H_{k,n,l}, 
\qquad k \in \K,\  n=1,2,\ldots,
\]
where $H_{k,n,l}$ denotes the service time of the $l$th customer in
class $k$ who arrives at the $n$th slot.  We assume that service times
of class $k$ ($k \in \K$) customers are i.i.d.\ according to a general
distribution. Let $H_k$ ($k \in \K$) denote a generic random variable
for $H_{k,n,l}$.  For simplicity, we assume that $H_k$ is positive.
The probability generating function of $H_k$ ($k \in \K$) is denoted
by $H_k(z)$:
\[
H_k(z) = \E\left[ z^{H_k} \right] 
= \sum_{m=1}^{\infty} \Pr(H_k = m) z^m,
\quad
k \in \K.
\]
Let $X_n$ ($n=0,1,\ldots$) denote the total amount of unfinished work
at the $n$th slot. We assume that $\{X_n\}_{n=0,1,\ldots}$ satisfies
\begin{equation}
X_n = \max(X_{n-1} -1, 0) + B_n, 
\qquad n=1,2,\ldots,
\label{X_n}
\end{equation}
where 
\[
B_n = \sum_{k \in \K} B_{k,n}, 
\qquad n=1,2,\ldots.
\]
Equation (\ref{X_n}) implies that the service discipline is
work-conserving and that $B_n=0$ if $X_n=0$. Because $B_n=0$ is
equivalent to $A_{k,n} = 0$ for all $k$ ($k \in \K$), we have
\begin{equation}
\Pr( S_n=0 \mid X_n=0) 
= \prod_{k \in \K} \Pr( S_{k,n}=0 \mid X_n=0) 
= 1,
\label{Q=S}
\end{equation}
where the first equality follows from the independence of arrival
streams and Assumption \ref{assumption-2}.

Within the above settings, we consider the preemptive-resume and
nonpreemptive priority disciplines, where priority is assigned in a
descending order of class indices. Specifically, class 1 customers
have the highest priority and class $K$ customers the lowest.  In
these priority disciplines, services of class $k$ customers start only
when there are no customers in classes 1 to $k-1$. Once a
service starts in the nonpreemptive priority queue, it continues to
the completion. In the preemptive-resume priority queue, however,
services are interrupted when customers in higher priority classes
arrive, and they restart from the interrupted points again when all
customers in higher priority classes complete their services. The order
of services within each class is arbitrary, provided that it is
independent of service times of customers and once a service starts,
it completes before other services in the same class start.

We define $\bm{A}_k(z)$ ($k \in \K$) as an $(M_k+1) \times (M_k+1)$
matrix whose ($i,j$)th ($i,j \in \M_k$) element $A_{k,i,j}(z)$ is
given by
\[
A_{k,i,j}(z) = P_{k,i,j} a_{k,i,j}(z), 
\]
where $P_{k,i,j}$ denotes the ($i,j$)th element of $\bm{P}_k$ and
\[
a_{k,i,j}(z) = \E\left[z^{A_{k,n}} \mid 
S_{k,n-1} = i, S_{k,n} =j \right].
\]
Let $\alpha_{k,j}$ and $T_{k,i,j}$ ($k \in \K$, $i,j \in
\M_k^{[\on]}$) denote the $j$th element of $\bm{\alpha}_k$ and the
($i,j$)th element of $\bm{T}_k$.  It then follows that
\begin{equation}
\bm{A}_k(z) = \left(\begin{array}{cc}
p_k & (1-p_k) \bm{\alpha}_k(z)
\\
(\bm{I}_k - \bm{T}_k)\bm{e}_k & \bm{T}_k(z)
\end{array}\right), 
\qquad
k \in \K,
\label{A_k(z)}
\end{equation}
where $\bm{\alpha}_k(z)$ denotes a $1 \times M_k$ vector whose $j$th
($j \in \M_k^{[\on]}$)
element $\alpha_{k,j}(z)$ is given by
\[
\alpha_{k,j}(z) = \alpha_{k,j} a_{k,0,j}(z),
\]
and $\bm{T}_k(z)$ denotes an $M_k \times M_k$ matrix whose $(i,j)$th
($i,j \in \M_k^{[\on]}$) element $T_{k,i,j}(z)$ is given by
\[
T_{k,i,j}(z) = T_{k,i,j} a_{k,i,j}(z).
\]
Therefore $\{A_{k,n}\}_{n=1,2,\ldots}$ is characterized fully by
$p_k$, $\bm{\alpha}_k(z)$, and $\bm{T}_k(z)$. We also define
$\bm{B}_k(z)$ ($k \in \K$) as
\[
\bm{B}_k(z) = \bm{A}_k(H_k(z)).
\]
Note that the ($i,j$)th ($i,j \in \M_k$) element $B_{k,i,j}(z)$ of
$\bm{B}_k(z)$ is given by
\[
B_{k,i,j}(z) = \E\left[z^{B_{k,n}} \one(S_{k,n} = j) 
\mid S_{k,n-1} = i \right], 
\]
where $\one(\chi)$ denotes the indicator function of event $\chi$.

Let $\bm{\pi}_k$ ($k \in \K$) denote the steady-state probability
vector of the underlying Markov chain for class $k$. Because the
underlying Markov chain for each class is assumed to be
irreducible, $\bm{\pi}_k$ is determined uniquely by $\bm{\pi}_k =
\bm{\pi}_k \bm{P}_k$ and $\bm{\pi}_k \check{\bm{e}}_k=1$, i.e.,
\[
\bm{\pi}_k = \left( \pi_{k,0},\ \bm{\pi}_{k,1} \right),
\quad
k \in \K,
\]
where $\check{\bm{e}}_k$ ($k \in \K$) denotes an $(M_k+1) \times 1$ vector
whose elements are all equal to one, and 
\[
\pi_{k,0} 
=
\frac{(1-p_k)^{-1}}{(1-p_k)^{-1} + \bm{\alpha}_k 
( \bm{I}_k - \bm{T}_k )^{-1}\bm{e}_k},
\qquad
\bm{\pi}_{k,1} 
=
\frac{\bm{\alpha}_k ( \bm{I}_k-\bm{T}_k )^{-1}}
{(1-p_k)^{-1} + \bm{\alpha}_k 
( \bm{I}_k-\bm{T}_k )^{-1}\bm{e}_k},
\quad
k \in \K.
\]
Let $\lambda_k$ ($k \in \K$) denote the arrival rate of class $k$
customers.
\[
\lambda_k = 
\bm{\pi}_k \bm{A}_k^{(1)} \bm{e}_k
=
\frac{\bm{\alpha}_k^{(1)}\bm{e}_k
+
\bm{\alpha}_k ( \bm{I}_k-\bm{T}_k )^{-1}
\bm{T}_k^{(1)}\bm{e}_k}
{(1-p_k)^{-1} + \bm{\alpha}_k 
( \bm{I}_k-\bm{T}_k )^{-1}\bm{e}_k},
\quad
k \in \K.
\]
We define $\lambda$ as 
\[
\lambda = \sum_{k \in \K} \lambda_k.
\]
Let $\rho_k$ ($k \in \K$) denote the traffic intensity of class $k$
customers.  We then have
\[
\rho_k = \bm{\pi}_k\bm{B}_k^{(1)} \check{\bm{e}}_k
= \lambda_k \E[H_k],
\quad
k \in \K.
\]

\begin{assumption}\label{assumption-3}
\begin{equation}
\rho = \sum_{k \in \K} \rho_k < 1,
\label{rho}
\end{equation}
and the system is in steady state. 
\end{assumption}

\section{Conservation law and busy cycle}\label{section:law}

The priority queues described in Section~\ref{section:model} can be
viewed as special cases of the model studied in \cite{Taki05}, where
the mean amount of unfinished work in the discrete-time
work-conserving single-server queue with multiple, independent arrival
streams with geometrically distributed idle periods is derived. In
what follows, we briefly review the result in \cite{Taki05} and then
establish the conservation law for nonpreemptive services.
Furthermore, we derive explicit expressions for the first two
factorial moments of busy cycle lengths.

\subsection{The amount of unfinished work and the conservation law}

It is clear from Assumption \ref{assumption-1} and (\ref{X_n}) that
the multivariate process $\{(X_n$, $S_{1,n}$, $S_{2,n}$, $\ldots,
S_{K,n})\}_{n=0,1,\ldots}$ forms a discrete-time Markov chain.  We
arrange $(S_{1,n}, S_{2,n},\ldots,S_{K,n})$ in lexicographic order and
it is denoted by $S_n$. Note that $S_n$ takes a value in $\M = \{0, 1,
\ldots, M\}$, where $M= \prod_{k \in \K} (1+M_k) - 1$. It then follows
from (\ref{X_n}) that the bivariate process $\{(X_n, S_n);
n=0,1,\ldots\}$ forms a Markov chain of M/G/1 type \cite{Neut89}. To
see this, we define $\bm{B}^*(z)$ as an $(M+1) \times (M+1)$ matrix
given by
\[
\bm{B}^*(z)
=
\bm{B}_1(z) \otimes \bm{B}_2(z) \otimes \cdots \otimes
\bm{B}_K(z),
\]
where $\otimes$ stands for Kronecker product \cite{Grah81}.  Furthermore,
let $\bm{B}(m)$ ($m=0,1,\ldots$) denote the coefficient matrix of
$z^m$ in $\bm{B}(z)$.
\[
\bm{B}^*(z) = \sum_{m=0}^{\infty} \bm{B}(m) z^m .
\]
Roughly speaking, $\bm{B}(m)$ represents the one-step transition
probability matrix of the product underlying Markov chain when $m$
units of work arrive. Since $\{X_n\}_{n=0,1,\ldots}$ is governed by
(\ref{X_n}), the transition probability matrix of the irreducible
Markov chain $\{(X_n, S_n); n=0,1,\ldots\}$ takes the following form:
\[
\left(
\begin{array}{ccccc}
\bm{B}(0) & \bm{B}(1) & \bm{B}(2) & \bm{B}(3) & \cdots \\
\bm{B}(0) & \bm{B}(1) & \bm{B}(2) & \bm{B}(3) & \cdots \\
\bm{O}    & \bm{B}(0) & \bm{B}(1) & \bm{B}(2) & \cdots \\
\bm{O}    & \bm{O}    & \bm{B}(0) & \bm{B}(1) & \cdots \\
\bm{O}    & \bm{O}    & \bm{O}    & \bm{B}(0) & \cdots \\
\vdots    & \vdots    & \vdots    & \vdots    & \ddots
\end{array}
\right).
\]
Note that the irreducible Markov chain $\{(X_n, S_n)\}_{n=0,1,\ldots}$
is positive recurrent when (\ref{rho}) holds (Theorem 3.2.1 of
\cite{Neut89}). Let $\bm{x}(z)$ denote a $1\times (M+1)$ vector whose
$i$th ($i \in \M$) element $x_i(z)$ represents
\[
x_i(z) = \lim_{n \to \infty} 
\E \left[ z^{X_n} \one(S_n =i) \right].
\]
Note here that owing to Assumption \ref{assumption-2},
$\bm{B}(0)$ takes a form:
\begin{equation}
\bm{B}(0) = \left(
\begin{array}{ccccc}
B_{0,0}(0) & 0 & 0 &  \cdots & 0 \\
B_{1,0}(0) & 0 & 0 &\cdots & 0 \\
B_{2,0}(0) & 0 & 0 & \cdots & 0 \\
\vdots & \vdots & \vdots  & \ddots & \vdots  \\
B_{M,0}(0) & 0 & 0 & \cdots & 0
\end{array}
\right) ,
\label{B(0)}
\end{equation}
and therefore $\bm{x}(0) = (1-\rho) \bm{g}$ \cite{Taki05}, where
$\bm{g}$ denotes a $1 \times (M+1)$ unit vector given by
\begin{equation}
\bm{g} = (1,0,0,\ldots,0).
\label{eq:bm{g}}
\end{equation}
See (\ref{Q=S}) also. It then follows from the standard
matrix-analytic method \cite{Neut89} that
\[
\bm{x}(z) [z \bm{I} - \bm{B}^*(z) ] 
= (1-\rho) (z-1) \bm{g} \bm{B}^*(z),
\]
where $\bm{I}$ denotes an $(M+1)\times(M+1)$ unit matrix.

We define $\delta_k(z)$ ($k \in \K$) as the Perron-Frobenius
eigenvalue of $\bm{A}_k(z)$, and let $\bm{u}_k(z)$ and $\bm{v}_k(z)$
($k \in \K$) denote the left and right eigenvectors of $\bm{A}_k(z)$
associated with $\delta_k(z)$. We assume that $\bm{u}_k(z) \check{\bm{e}}_k =
\bm{u}_k(z) \bm{v}_k(z) = 1$, so that $\bm{u}_k(z)$ and $\bm{v}_k(z)$
are uniquely determined.

\begin{lemma}[Theorem A.2.1 in \cite{Neut89}]\label{lemma-Neut89} 
$\bm{u}_k(1)$ ($k \in \K$) and $\bm{v}_k(1)$ ($k \in \K$) are given by
\begin{equation}
\bm{u}_k(1) = \bm{\pi}_k, 
\qquad
\bm{v}_k(1) = \bm{e}_k,
\quad
k \in \K,
\label{u_k-v_k}
\end{equation}
and $\delta_k^{(1)}$ ($k \in \K$) is given by
\begin{equation}
\delta_k^{(1)} = \lambda_k,
\quad
k \in \K.
\label{delta_k^1}
\end{equation}
\end{lemma}

We define $C_k$ ($k \in \K$) as a generic random variable representing
the length of a randomly chosen active period of the class $k$ arrival
stream, and $\Lambda_k$ ($k \in \K$) as a generic random variable
representing the total number of class $k$ customers arriving in a
randomly chosen active period of class $k$. Furthermore, let
$\Lambda_k(\tilde{C}_k)$ ($k \in \K$) denote a generic random variable
representing the total number of class $k$ customers arriving in the
forward recurrence time of a randomly chosen active period of class
$k$. It then follows from (\ref{P_k}) and
(\ref{A_k(z)}) that for $k \in \K$, 
\begin{align*}
\E\big[z^{C_k} \omega^{\Lambda_k} \big]
&=
z \bm{\alpha}_k(\omega) \big[\bm{I}_k-z\bm{T}_k(\omega)\big]^{-1} 
(\bm{I}_k-\bm{T}_k)\bm{e}_k,
\\
\E\big[z^{\tilde{C}_k}\big]
&=
\frac{1 - \E\big[z^{C_k}\big]}{\E[C_k] (1-z)}
=
\frac{\bm{\alpha}_k (\bm{I}_k-\bm{T}_k)^{-1}}{\E[C_k]}
(\bm{I}_k-z \bm{T}_k)^{-1} (\bm{I}_k-\bm{T}_k)\bm{e}_k,
\\
\E\big[z^{\Lambda_k(\tilde{C}_k)}\big]
&=
\frac{\bm{\alpha}_k (\bm{I}_k-\bm{T}_k)^{-1}}{\E[C_k]}
\big(\bm{I}_k-\bm{T}_k(z)\big)^{-1} 
(\bm{I}_k-\bm{T}_k)\bm{e}_k,
\\
\E\big[z^{\tilde{\Lambda}_k}\big]
&=
\frac{1 - \E\big[z^{\Lambda_k}\big]}{\E[\Lambda_k] (1-z)},
\end{align*}
where 
\[
\E[C_k] 
= 
\bm{\alpha}_k (\bm{I}_k - \bm{T}_k)^{-1} \bm{e}_k,
\qquad
\E[\Lambda_k]
=
\bm{\alpha}_k^{(1)} \bm{e}_k
+
\bm{\alpha}_k (\bm{I}_k - \bm{T}_k)^{-1} \bm{T}_k^{(1)} \bm{e}_k.
\]
We then have for $k \in \K$,
\begin{align*}
\E[C_k \Lambda_k]
&=
\bm{\alpha}_k^{(1)} 
(\bm{I}_k - \bm{T}_k)^{-1} \bm{e}_k
+
\bm{\alpha}_k (\bm{I}_k - \bm{T}_k)^{-1} \bm{T}_k^{(1)} 
(\bm{I}_k - \bm{T}_k)^{-1} \bm{e}_k
+
\bm{\alpha}_k
(\bm{I}_k - \bm{T}_k)^{-2} \bm{T}_k^{(1)} \bm{e}_k,
\\
\E[\tilde{C}_k]
&=
\frac{\bm{\alpha}_k (\bm{I}_k-\bm{T}_k)^{-1}}{\E[C_k]}
(\bm{I}_k - \bm{T}_k)^{-1} \bm{T}_k \bm{e}_k,
\\
\E\big[\Lambda_k(\tilde{C}_k)\big]
&=
\frac{\bm{\alpha}_k (\bm{I}_k-\bm{T}_k)^{-1}}{\E[C_k]}
(\bm{I}_k - \bm{T}_k)^{-1} \bm{T}_k^{(1)} \bm{e}_k,
\\
\E[\tilde{\Lambda}_k]
&=
\frac{1}{2\E[\Lambda_k]}
\Big[
\bm{\alpha}_k^{(2)} \bm{e}_k
+
2 \bm{\alpha}_k^{(1)} (\bm{I}_k - \bm{T}_k)^{-1} 
\bm{T}_k^{(1)} \bm{e}_k
\\
&\qquad\qquad\qquad {} 
+
2 \bm{\alpha}_k (\bm{I}_k - \bm{T}_k)^{-1} \bm{T}_k^{(1)} 
(\bm{I}_k - \bm{T}_k)^{-1} \bm{T}_k^{(1)} \bm{e}_k
+
\bm{\alpha}_k (\bm{I}_k - \bm{T}_k)^{-1} 
\bm{T}_k^{(2)} \bm{e}_k
\Big].
\end{align*}
We define $v_{k,j}(z)$ ($k \in \K$, $j \in \M_k$) as the $j$th
element of $\bm{v}_k(z)$. 

\begin{lemma}[Lemma 1 in \cite{Taki05}]\label{lemma-Taki05}
$\delta_k^{(2)}$ ($k \in \K$) and $v_{k,0}^{(1)}$ ($k \in \K$) are
given by
\begin{eqnarray}
\delta_k^{(2)} &=& 
2 \lambda_k \E[\tilde{\Lambda}_k] 
- 2 \lambda_k \pi_k^{[\on]} 
\frac{\E[ C_k \Lambda_k]}{\E[C_k]}
+ 2 \lambda_k^2 \pi_k^{[\on]} \left(1+\E[\tilde{C}_k] \right),
\label{delta_k^2}
\\
v_{k,0}^{(1)} &=&
\lambda_k \pi_k^{[\on]}  \left(1+\E[\tilde{C}_k] \right) 
- \pi_k^{[\on]} \E[\Lambda_k(\tilde{C}_k)],
\label{v_{k,0}^1}
\end{eqnarray}
where $\pi_k^{[\on]} = \bm{\pi}_{k,1} \bm{e}_k$ ($k \in
\K$) denotes the stationary probability of the class $k$ arrival
stream being active.
\end{lemma}

We define $\bm{u}(z)$, $\bm{v}(z)$, and $\delta(z)$ as
\begin{eqnarray*}
\bm{u}(z) &=& \bm{u}_1(H_1(z)) \otimes \bm{u}_2(H_2(z)) 
\otimes \cdots \otimes \bm{u}_K(H_K(z)),
\\
\bm{v}(z) &=& \bm{v}_1(H_1(z)) \otimes \bm{v}_2(H_2(z)) 
\otimes \cdots \otimes \bm{v}_K(H_K(z)),
\\
\delta(z) &=& \prod_{k \in \K} \delta_k(H_k(z)),
\end{eqnarray*}
respectively. Note here that
\[
\bm{u}(z)\bm{B}^*(z) = \delta(z) \bm{u}(z),
\qquad
\bm{B}^*(z) \bm{v}(z) = \delta(z)\bm{v}(z),
\qquad
\bm{u}(z) \bm{v}(z) = \bm{u}(z) \bm{e} = 1 ,
\]
\[
\bm{u}(1) = \bm{\pi}_1 \otimes  \bm{\pi}_2  \otimes  
\cdots  \otimes \bm{\pi}_K ,
\qquad 
\bm{v}(1) = \bm{e}, 
\qquad 
\delta(1) = 1  ,
\]
and
\begin{equation}
\delta^{(1)} = \sum_{k \in \K} \lambda_k \E[H_k]
=\rho,
\qquad
\delta^{(2)}
=
\sum_{k \in \K} \left(
2 \rho_k \E[\tilde{H}_k]
+ \rho_k (\rho - \rho_k) 
+ \E[H_k]^2 \delta_k^{(2)} 
\right) ,
\label{delta-moment}
\end{equation}
where $\bm{e}$ denotes an $(M+1)\times 1$ vector whose elements are
all equal to one.

Let $U_k$ ($k \in \K$) denote a generic random variable representing
the total amount of unfinished work of class $k$ in steady state, and
$U$ denote a generic random variable representing the stationary
amount of unfinished work, i.e.,
\[
U = \sum_{k \in \K} U_k.
\]
Note that $U$ is considered as a generic random variable for $\{X_n;\
n=0,1, \ldots\}$ in steady state. Applying Eq.\ (11) in \cite{Taki05} to
our priority queues, we immediately obtain the following lemma.
\begin{lemma}[Section 2.2 in \cite{Taki05}]\label{lemma-unfinished}
In any work-conserving single-server queue 
satisfying Assumptions~\ref{assumption-1} to \ref{assumption-3}, 
the mean amount $\E[U]$ of unfinished work in steady state is given by
\begin{equation}
\E[U] =
\rho 
+ \frac{\ds\sum_{k \in \K} \rho_k \E[\tilde{H}_k]}{1-\rho} 
+ \frac{\ds\sum_{k \in \K} \rho_k(\rho -\rho_k)}{2(1-\rho)} 
+ \frac{\ds\sum_{k \in \K} \E[H_k]^2 \delta_k^{(2)}}{2(1-\rho)} 
+ \sum_{k \in \K} \E[H_k] v_{k,0}^{(1)} ,
\label{eq:lemma-unfinished}
\end{equation}
where $\delta_k^{(2)}$ and $v_{k,0}^{(1)}$ are given 
by (\ref{delta_k^2}) and (\ref{v_{k,0}^1}). 
\end{lemma}

\begin{theorem}\label{theorem-conservation}
In any work-conserving single-server queue with nonpreemptive
services, satisfying Assumptions~\ref{assumption-1} to
\ref{assumption-3}, the mean waiting time $\E[W_k]$ ($k \in \K$) of
class $k$ customers in steady state satisfies
\begin{eqnarray*}
\sum_{k \in \K} \rho_k \E[W_k] 
&=&
\frac{\rho\ds\sum_{k \in \K} \rho_k \E[\tilde{H}_k]}{1-\rho} 
+ \frac{\ds\sum_{k \in \K} \rho_k(\rho -\rho_k)}{2(1-\rho)} 
+ \frac{\ds\sum_{k \in \K} 
\rho_k \E[H_k] \left( 
\E[\tilde{\Lambda}_k] 
- 
\pi_k^{[\on]} \ds\frac{\E[C_k \Lambda_k ]}{\E[C_k]}
\right)}{1-\rho} 
\nonumber
\\
&& 
\qquad {}
+ 
\sum_{k \in \K} \pi_k^{[\on]} \rho_k 
\left( 1 + \frac{\rho_k}{1-\rho} \right)
\left(1 + \E[\tilde{C}_k] \right)
-
\sum_{k \in \K} \pi_k^{[\on]} \E[H_k] \E[\Lambda_k(\tilde{C}_k)].
\end{eqnarray*}
\end{theorem}

\begin{proof}
Let $L_k^q$ ($k \in \K$) and $W_k$ ($k \in \K$) denote generic random
variables representing the number of waiting customers of class $k$
and the waiting time of a randomly chosen class $k$ customer in steady
state. Recall that services are nonpreemptive and the order of
services within each class is independent of service times. It then
follows from Little's law that
\[
\E[U] 
= 
\sum_{k \in \K} \E[U_k]
=
\sum_{k \in \K} 
\left\{ \E[H_k] \E[L_k^q] + \rho_k (\E[\tilde{H}_k]+1) \right\}
=
\sum_{k \in \K} 
\left\{ \rho_k \E[W_k] +\rho_k (\E[\tilde{H}_k] + 1) \right\},
\]
and therefore we have
\begin{equation}
\sum_{k \in \K} \rho_k \E[W_k] 
=
\E[U]-  \sum_{k \in \K} \rho_k (\E[\tilde{H}_k]+1)
=
\E[U]-  \rho - \sum_{k \in \K} \rho_k \E[\tilde{H}_k].
\label{work-conserving-setup}
\end{equation}
The theorem now follows from Lemmas \ref{lemma-Taki05} and 
\ref{lemma-unfinished}, and (\ref{work-conserving-setup}).
\end{proof}

\begin{remark}
For $K=1$, Lemma \ref{lemma-unfinished} and 
Theorem \ref{theorem-conservation} are reduced to
\begin{align}
\E[U_1] 
&=
\rho_1
+ 
\frac{\rho_1 \E[H_1]}{1-\rho_1} 
\left( 
1 
+
\E[\tilde{\Lambda}_1] 
- 
\pi_1^{[\on]} \ds\frac{\E[C_1 \Lambda_1 ]}{\E[C_1]}
\right)
+ 
\rho_1 \pi_1^{[\on]} \left(
\frac{1 + \E[\tilde{C}_1]}{1-\rho_1} 
-
\frac{\E[\Lambda_1(\tilde{C}_1)]}{\lambda_1} \right),
\label{eq:K=1-U}
\\
\E[W_1] 
&=
\frac{\E[H_1]}{1-\rho_1} 
\left( 
\rho_1 
+
\E[\tilde{\Lambda}_1] 
- 
\pi_1^{[\on]} \ds\frac{\E[C_1 \Lambda_1 ]}{\E[C_1]}
\right)
+
\pi_1^{[\on]} \left(
\frac{1 + \E[\tilde{C}_1]}{1-\rho_1} 
-
\frac{\E[\Lambda_1(\tilde{C}_1)]}{\lambda_1} \right).
\label{K=1}
\end{align}
\end{remark}

\subsection{Moments of busy cycle lengths}\label{sec-busy}

We consider the length of a randomly chosen busy cycle in the
work-conserving single-server queue satisfying Assumptions
\ref{assumption-1} to \ref{assumption-3}, 
where a busy cycle is defined as the first recurrence time to the empty
system.  Let $F$ denote the first passage time to the empty system.
\[
F = \inf \{n \geq 1;\ X_n = 0\} .
\]
The probability generating function for the lengths of busy cycles in
steady state is then given by $\E[z^F \mid (X_0,S_0)=(0,0)]$ (see
(\ref{Q=S})).  Note that the busy cycle length is equal to one if no
arrivals happen at the first slot, and otherwise it is constituted of
one slot, at which the system is empty, and successive services of
customers.  As we will see, $F$ plays an important role in the
analysis.

Let $\bm{G}(z)$ denote an $(M+1)\times (M+1)$ matrix whose $(i,j)$th
($i,j \in \M$) element represents $\E[z^F \one(S_F=j) \mid
  (X_0,S_0)=(1,i)]$. It then follows from the standard matrix-analytic
method~\cite{Neut89} that
\begin{equation}
\bm{G}(z)
= 
z \sum_{m=0}^{\infty} \bm{B}(m) \bm{G}^m(z).
\label{ch4:4}
\end{equation}
Because $\E[z^F \one(S_F=j) \mid (X_0,S_0)=(0,i)] =
\E[z^F \one(S_F=j) \mid (X_0,S_0)=(1,i)]$, we have
\[
\E[z^F \mid (X_0,S_0)=(0,0)] = \bm{g} \bm{G}(z) \bm{e},
\]
where $\bm{g}$ is given by (\ref{eq:bm{g}}).  On the other hand, owing
to (\ref{B(0)}) (or Assumption \ref{assumption-2}), $\bm{G}(z)$ takes
the following form:
\begin{equation}
\bm{G}(z) = \left(
\begin{array}{ccccc}
f_0(z) & 0 & 0 &  \cdots & 0 \\
f_1(z) & 0 & 0 &\cdots & 0 \\
f_2(z) & 0 & 0 & \cdots & 0 \\
\vdots & \vdots & \vdots  & \ddots & \vdots  \\
f_M(z) & 0 & 0 & \cdots & 0
\end{array}
\right) .
\label{G(z)-form}
\end{equation}
Thus the probability generating function for the lengths of  
busy cycles in steady state is given by $f_0(z)$. 

\begin{theorem}\label{theorem-f}
In any work-conserving single-server queue satisfying
Assumptions~\ref{assumption-1} to \ref{assumption-3}, the first two
factorial moments, $f_0^{(1)}$ and $f_0^{(2)}$, of the lengths of busy
cycles in steady state are given by
\begin{eqnarray*}
f_0^{(1)} 
&=& \frac{1}{1-\rho} ,
\\
f_0^{(2)}
&=&
\frac{2 \rho}{(1-\rho)^2}
+ 
\frac{1}{(1-\rho)^3} \sum_{k \in K} \left(
2 \rho_k \E[\tilde{H}_k] 
+ \rho_k (\rho - \rho_k) 
+ \E[H_k]^2 \delta_k^{(2)} 
\right),
\end{eqnarray*}
where $\delta_k^{(2)}$ is given in (\ref{delta_k^2}).
\end{theorem}

\begin{remark}
It is known that $\bm{g}\bm{G}^{(1)}\bm{e}= (1-\rho)^{-1}$ in general
Markov chains of M/G/1 type (see eq.(3.1.14) in \cite{Neut89}).
\end{remark}

\begin{proof}
We define $\bm{f}(z)$ as an $(M+1) \times 1$ vector whose $j$th ($j
\in \M$) element is given by $f_j(z)$ in (\ref{G(z)-form}).  We then
have
\[
\bm{G}(z)= \bm{f}(z) \bm{g}, \qquad \bm{g} \bm{f}(z) = f_0(z),
\]
where $\bm{g}$ is given by (\ref{eq:bm{g}}). 
Substituting the aboves into (\ref{ch4:4}), we obtain
\begin{equation}
\bm{f}(z) \bm{g} 
=
z \bm{B}(0)
+ z \sum_{m=1}^{\infty} \bm{B}(m) \bm{f}(z) 
[f_0(z)]^{m-1} \bm{g} .
\label{f(z)g}
\end{equation}
Furthermore, post-multiplying both sides of (\ref{f(z)g}) by $\bm{e}$
yields
\begin{eqnarray*}
\bm{f}(z) &=& 
z \bm{B}(0) \bm{e}
+ z \sum_{m=1}^{\infty} \bm{B}(m)
[f_0(z)]^{m-1} \bm{f}(z) 
\\
&=&  
z \bm{B}(0) \bm{e}
+ \frac{z}{f_0(z)} 
( \bm{B}^*(f_0(z)) - \bm{B}(0) ) 
\bm{f}(z) ,
\end{eqnarray*}
and therefore we have 
\begin{equation}
[ f_0(z) \bm{I} - z  \bm{B}^*(f_0(z)) ] \bm{f}(z) 
% =
% z f_0(z) \bm{B}(0) \bm{e} - z \bm{B}(0) \bm{f}(z) 
= {\bf 0},
\label{ch4:6}
\end{equation}
where we use $\bm{B}(0) \bm{f}(z) = f_0(z) \bm{B}(0) \bm{e}$,
which comes from (\ref{B(0)}).

Pre-multiplying both sides of (\ref{ch4:6}) by $\bm{u}(f_0(z))$
yields
\begin{equation}
[f_0(z) - z \delta(f_0(z)) ] 
\bm{u}( f_0(z) ) \bm{f}(z)
= 0.
\label{ch4:9}
\end{equation}
Differentiating both sides of (\ref{ch4:9}) and taking the
limit $z \to 1-$, we obtain
\[
f_0^{(1)} - 1  - \delta^{(1)} f_0^{(1)} = 0,
\qquad 
f_0^{(2)} - 2 \delta^{(1)} f_0^{(1)}
- \delta^{(2)} (f_0^{(1)})^2 
- \delta^{(1)} f_0^{(2)} 
= 0, 
\]
from which it follows that
\begin{equation}
f_0^{(1)} 
= \frac{1}{1-\delta^{(1)}},
\qquad
f_0^{(2)}
=
\frac{2 \delta^{(1)} f_0^{(1)} 
+ 
\delta^{(2)}(f_0^{(1)})^2} {1-\delta^{(1)}}  .
\label{f-1-2}
\end{equation}
Substituting (\ref{delta-moment}) into (\ref{f-1-2}) completes
the proof.
\end{proof}

For later use, we provide the following corollary, which is an
immediate consequence of Theorem \ref{theorem-f}.  Consider a
stationary work-conserving single-server queue fed only by class 1 to
class $k-1$ ($k \in \K$) arrival streams, satisfying
Assumptions~\ref{assumption-1} to \ref{assumption-3}. We define
$f_{k-1,0}^{(1)}$ and $f_{k-1,0}^{(2)}$ as the first and second
factorial moments of the lengths of busy cycles in this queue.  Let
\[
\K_k^+ = \{1,2,\ldots,k\},
\qquad
\K_k^- = \{k,k+1,\ldots,K\},
\quad
k \in \K,
\]
and we define $\rho_k^+$ ($k \in \K$) as
\[
\rho_k^+ = \sum_{m \in \K_k^+} \rho_m.
\]
For the sake of convenience, let $\rho_0^+ = 0$.

\begin{corollary}\label{corollary-f}
$f_{k-1,0}^{(1)}$ and
$f_{k-1,0}^{(2)}$ ($k \in \K_2^-$) are given by
\begin{align}
f_{k-1,0}^{(1)} 
&= 
\frac{1}{1-\rho_{k-1}^+},
\label{eq:f_k-1,0^(1)} 
\\
f_{k-1,0}^{(2)}
&=
\frac{2 \rho_{k-1}^+}{(1-\rho_{k-1}^+)^2}
+ 
\frac{1}{(1-\rho_{k-1}^+)^3} 
\sum_{\ell \in \K_{k-1}^+} 
\left(
2 \rho_{\ell} \E[\tilde{H}_{\ell}] 
+ \rho_{\ell} (\rho_{k-1}^+ - \rho_{\ell}) 
+ \E[H_{\ell}]^2 \delta_{\ell}^{(2)} 
\right),
\nonumber
\end{align}
where $\delta_k^{(2)}$ is given in (\ref{delta_k^2}).
\end{corollary}

\section{The preemptive-resume priority queue}
\label{sec-preemptive}

In this section, we consider the mean waiting time $\E[W_k^{\PR}]$ ($k
\in K$) of class $k$ customers in the preemptive-resume priority queue
in steady state.  Owing to the preemptive-resume priority,
$\E[W_k^{\PR}]$ ($k \in \K_{K-1}^+$) is independent of lower
classes. In particular, the mean waiting time of class 1 customers is
identical to the mean waiting time in the case of $K=1$. Therefore,
$\E[W_1^{\PR}]$ is identical to $\E[W_1]$ in (\ref{K=1}).  In what
follows we consider class $k$ ($k \in \K_2^-$).

We define $H_k^{\PR}$ ($k \in K_2^-$) as a random variable
representing the length of the interval from the beginning of a class
$k$ service to its completion, which we call the service completion
time of class $k$. 

\begin{lemma}\label{lemma:mean_completion}
The mean service completion time of class $k$
($k \in \K_2^-$) is given by
\begin{equation}
\E[H_k^{\PR}] 
=
\frac{\E[H_k]-\rho_{k-1}^+}{1-\rho_{k-1}^+},
\quad
k \in \K_2^-.
\label{eq:E[H_k^PR]}
\end{equation}
\end{lemma}

\begin{proof}
Note that if $H_k=\ell$ ($\ell=1,2,\ldots$), at most $\ell-1$ service
interruptions can occur. See the upper panel in
Fig.\ \ref{figure:completion}.  Let $\{F_{k-1,m}\}_{m=1,2,\ldots}$
denote a sequence of i.i.d.\ random variables with whose distribution
is given by the distribution of the busy cycle in the system only with
class 1 to class $k-1$. It then follows that
\[
H_k^{\PR} 
=
\sum_{m=1}^{H_k-1} F_{k-1,m} + 1,
\quad
k=2,3,\ldots,K.
\]
Note here that (cf.\ Corollary \ref{corollary-f})
\[
\E[F_{k-1.m}] 
=
f_{k-1,0}^{(1)}
=
\frac{1}{1-\rho_{k-1}^+},
\quad 
k=2,3,\ldots,K,\ m=1,2,\ldots.
\]
We thus have
\[
\E[H_k^{\PR}] 
=
\E\bigg[ \sum_{m=1}^{H_k-1} F_{k-1,m} \bigg] + 1
=
\frac{\E[H_k]-1}{1-\rho_{k-1}^+} + 1,
\quad
k = 2,3,\ldots,K,
\]
from which the lemma follows.
\end{proof}

\begin{figure}
\centering
\begin{picture}(110,21)(0,-6)
\put(5,0){\line(1,0){100}}
\put(10,10){\line(1,0){90}}
\put(10,0){\line(0,1){10}}
\put(20,0){\line(0,1){10}}
\put(40,0){\line(0,1){10}}
\put(50,0){\line(0,1){10}}
\put(60,0){\line(0,1){10}}
\put(70,0){\line(0,1){10}}
\put(90,0){\line(0,1){10}}
\put(100,0){\line(0,1){10}}
\put(25,13){\makebox(0,0){$F_{k-1,1}$}}
\put(31,13){\vector(1,0){9}}
\put(19,13){\vector(-1,0){9}}
\put(10,11){\line(0,1){4}}
\put(40,11){\line(0,1){4}}
\multiput(10,2)(0,2){4}{\line(1,0){10}}
\put(50,13){\makebox(0,0){$F_{k-1,2}$}}
\put(56,13){\vector(1,0){4}}
\put(44,13){\vector(-1,0){4}}
\put(60,11){\line(0,1){4}}
\multiput(40,2)(0,2){4}{\line(1,0){10}}
\put(75,13){\makebox(0,0){$F_{k-1,3}$}}
\put(81,13){\vector(1,0){9}}
\put(69,13){\vector(-1,0){9}}
\put(90,11){\line(0,1){4}}
\multiput(60,2)(0,2){4}{\line(1,0){10}}
\put(95,13){\makebox(0,0){$1$}}
\put(97,13){\vector(1,0){3}}
\put(93,13){\vector(-1,0){3}}
\put(100,11){\line(0,1){4}}
\put(55,-4){\makebox(0,0){$H_k^{\PR}$}}
\put(60,-4){\vector(1,0){40}}
\put(50,-4){\vector(-1,0){40}}
\put(10,-6){\line(0,1){4}}
\put(100,-6){\line(0,1){4}}
\multiput(90,2)(0,2){4}{\line(1,0){10}}
\end{picture}

\begin{picture}(110,56)(0,-9)
\put(5,0){\line(1,0){100}}
\put(10,0){\line(0,1){40}}
\put(10,40){\line(1,0){10}}
\put(20,30){\line(0,1){10}}
\put(20,30){\line(1,0){30}}
\put(50,20){\line(0,1){10}}
\put(50,20){\line(1,0){20}}
\put(70,10){\line(0,1){10}}
\put(70,10){\line(1,0){30}}
\put(100,0){\line(0,1){10}}
\multiput(20,0)(0,4){9}{\line(0,1){2}}
\multiput(50,0)(0,4){6}{\line(0,1){2}}
\multiput(70,0)(0,4){3}{\line(0,1){2}}
\put(2,20){\makebox(0,0){$H_k=4$}}
\put(3,23){\vector(0,1){17}}
\put(3,17){\vector(0,-1){17}}
\put(1,0){\line(1,0){4}}
\put(1,40){\line(1,0){4}}
\put(15,43){\makebox(0,0){1}}
\put(17,43){\vector(1,0){3}}
\put(13,43){\vector(-1,0){3}}
\put(10,41){\line(0,1){4}}
\put(20,41){\line(0,1){4}}
\multiput(10,30)(0,2){5}{\line(1,0){10}}
\put(35,33){\makebox(0,0){$F_{k-1,1}'$}}
\put(41,33){\vector(1,0){9}}
\put(29,33){\vector(-1,0){9}}
\put(50,31){\line(0,1){4}}
\multiput(40,20)(0,2){5}{\line(1,0){10}}
\put(40,20){\line(0,1){10}}
\put(60,23){\makebox(0,0){$F_{k-1,2}'$}}
\put(66,23){\vector(1,0){4}}
\put(54,23){\vector(-1,0){4}}
\put(70,21){\line(0,1){4}}
\multiput(60,10)(0,2){5}{\line(1,0){10}}
\put(60,10){\line(0,1){10}}
\put(85,13){\makebox(0,0){$F_{k-1,3}'$}}
\put(91,13){\vector(1,0){9}}
\put(79,13){\vector(-1,0){9}}
\put(100,11){\line(0,1){4}}
\multiput(90,0)(0,2){5}{\line(1,0){10}}
\put(90,0){\line(0,1){10}}
\put(0,-6){\makebox(0,0){$\tilde{R}_k^{\PR}$}}
\put(15,-6){\makebox(0,0){4}}
\put(35,-6){\makebox(0,0){3}}
\put(60,-6){\makebox(0,0){2}}
\put(85,-6){\makebox(0,0){1}}
\qbezier[18](10,-9)(10,-5)(10,0)
\qbezier[18](20,-9)(20,-5)(20,0)
\qbezier[18](50,-9)(50,-5)(50,0)
\qbezier[18](70,-9)(70,-5)(70,0)
\qbezier[18](100,-9)(100,-5)(100,0)
\end{picture}
\caption{The service completion time $H_k^{\PR}$ and 
the remaining service time $\tilde{R}_k^{\PR}$ 
($H_k=4$).}
\label{figure:completion}
\end{figure}
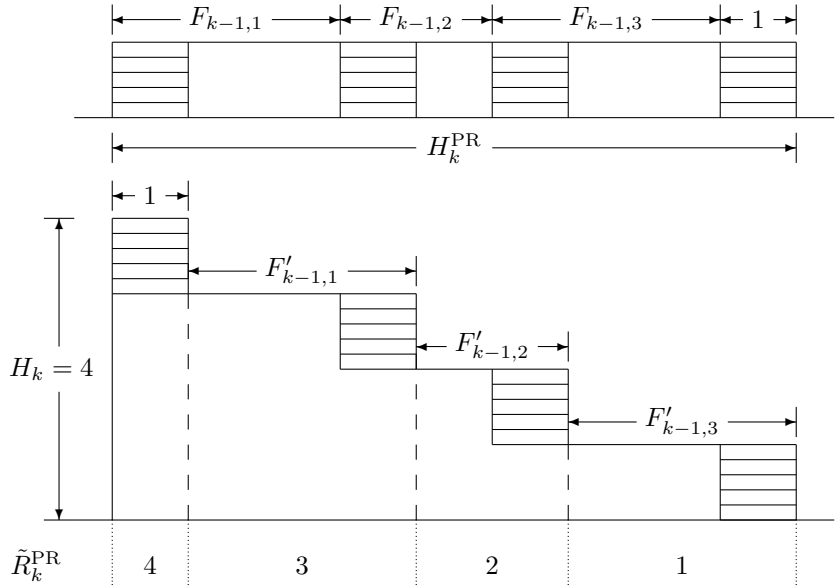

We define $\tilde{R}_k^{\PR}$ as the conditional remaining service
time (including one unit being served if any) of class $k$ given that
the service started but it has not yet completed.  See the lower panel
of Fig.\ \ref{figure:completion}.

\begin{lemma}\label{lemma:tilde{R}}
The mean remaining service time $\E\big[\tilde{R}_k^{\PR}\big]$ is
given by
\[
\E\big[\tilde{R}_k^{\PR}\big]
=
\frac{\E[H_k]}{\E\big[H_k^{\PR}\big]}
\bigg( 1 + \frac{\E\big[\tilde{H}_k\big]}{1-\rho_{k-1}^+} \bigg),
\quad
k \in \K_2^-.
\]
\end{lemma}

\begin{proof}
The standard inspection paradox argument yields 
(cf.\ Fig.\ \ref{figure:completion})
\[
\Pr\big(\tilde{R}_k^{\PR} = m\big)
=
\frac{\Pr(H_k=m)}{\E\big[H_k^{\PR}\big]}
+
\frac{\Pr(H_k > m)}{\big(1-\rho_{k-1}^+\big) \E\big[H_k^{\PR}\big]},
\quad
k \in \K_2^-,\ m=1,2,\ldots.
\]
It then follows that
\begin{align*}
\E\big[\tilde{R}_k^{\PR}\big]
&=
\frac{\E[H_k]}{\E\big[H_k^{\PR}\big]}
+
\frac{1}{\big(1-\rho_{k-1}^+\big) \E\big[H_k^{\PR}\big]}
\sum_{m=1}^{\infty} m \sum_{n=m+1}^{\infty} \Pr(H_k =n)
\\
&=
\frac{\E[H_k]}{\E\big[H_k^{\PR}\big]}
+
\frac{1}{\big(1-\rho_{k-1}^+\big) \E\big[H_k^{\PR}\big]}
\sum_{n=1}^{\infty} \Pr(H_k =n) \sum_{m=1}^{n-1} m
\\
&=
\frac{\E[H_k]}{\E\big[H_k^{\PR}\big]}
+
\frac{1}{\big(1-\rho_{k-1}^+\big) \E\big[H_k^{\PR}\big]}
\sum_{n=1}^{\infty} \frac{n(n-1)}{2} \Pr(H_k =n)
\\
&=
\frac{\E[H_k]}{\E\big[H_k^{\PR}\big]}
+
\frac{1}{\big(1-\rho_{k-1}^+\big) \E\big[H_k^{\PR}\big]}
\E[H_k] \E\big[\tilde{H}_k\big],
\quad
k \in \K_2^-,
\end{align*}
from which the lemma follows.
\end{proof}

We define $W_k^{\PR}$ and $U_k^{\PR}$ ($k \in \K$) as the
waiting time of a randomly chosen class $k$ customer and the amount of
unfinished work of class $k$ in steady state.  We also define
$Q_k^{\PR}$ ($k \in \K_2^-$) as the number of class $k$ customers who
have not yet started their services. Owing to the Little's law, a
class $k$ customer is in the middle of service with probability
$\lambda_k \E\big[H_k^{\PR}\big]$ in steady state and the contribution
of this customer to $U_k^{\PR}$ is given by $\tilde{R}_k^{\PR}$. It
then follows from Lemmas \ref{lemma:mean_completion} and
\ref{lemma:tilde{R}} that
\begin{align}
\E\big[U_k^{\PR}\big] 
&= 
\E[H_k] \E[Q_k^{\PR}]
+ 
\lambda_k \E[H_k^{\PR}] \E\big[\tilde{R}_k^{\PR}\big]
\nonumber
\\
&=
\rho_k \E\left[W_k^{\PR}\right]
+
\rho_k \bigg( 1 + \frac{\E[\tilde{H}_k]}{1-\rho_{k-1}^+} \bigg),
\quad
k \in \K_2^-,
\label{eq:U_k-W_k-PR}
\end{align}
where we use the Little's law $\E[Q_k^{\PR}] = \lambda_k \E[W_k^{\PR}]$ 
in the second equality.

\begin{lemma}\label{lemma:E[U_k^PR]}
The mean amount $\E[U_k^{\PR}]$ ($k \in \K$) of unfinished work
of class $k$ in the stationary preemptive-resume priority queue 
satisfying Assumptions~\ref{assumption-1} to \ref{assumption-3}
is given by
\begin{align}
\E[U_k^{\PR}] 
&=
\rho_k
+ 
\frac{\rho_k \ds\sum_{\ell \in \K_k^+} \rho_{\ell} \E\big[\tilde{H}_{\ell}\big]}
{(1-\rho_k^+)(1-\rho_{k-1}^+)}
+
\frac{\rho_k \E[\tilde{H}_k]}{1-\rho_{k-1}^+}
+ 
\ds\frac{\rho_k \ds\sum_{\ell \in \K_k^+} \rho_{\ell}(\rho_k^+ -\rho_{\ell})}
{2(1-\rho_k^+)(1-\rho_{k-1}^+)}
\nonumber
\\
&\qquad\qquad {}
+ 
\frac{\rho_k \rho_{k-1}^+}{1-\rho_{k-1}^+}
+ 
\frac{\rho_k \ds\sum_{\ell \in \K_k^+} \E[H_{\ell}]^2 \delta_{\ell}^{(2)}}
{2(1-\rho_k^+)(1-\rho_{k-1}^+)}
+ 
\frac{\E[H_k]^2 \delta_k^{(2)}}{2(1-\rho_{k-1}^+)}
+ 
v_{k,0}^{(1)} \E[H_k],
\quad
k \in \K,
\label{eq:lemma:E[U_k^PR]}
\end{align}
where $\rho_0^+=0$, and $\delta_k^{(2)}$ and $v_{k,0}^{(1)}$ are given
by (\ref{delta_k^2}) and (\ref{v_{k,0}^1}).
\end{lemma}

\begin{proof}
Note that $\E[U_1^{\PR}]$ equals $\E[U_1]$ in (\ref{eq:K=1-U}), and we
observe that Lemma \ref{lemma:E[U_k^PR]} holds for $k=1$ because
$\rho_0^+=0$. Let $U_{k+}^{\PR}$ ($k \in \K$) denote the amount of
unfinished work in classes 1 to $k$. In the preemptive-resume priority
queue, customers in class $k$ ($k \in \K_{K-1}^+$) are not affected by
the existence of customers in lower classes. Therefore,
$\E[U_{k+}^{\PR}]$ is given by
\[
\E\big[U_{k+}^{\PR}\big]
= 
\sum_{m=1}^k \E\big[U_m^{\PR}\big],
\quad
k \in \K,
\]
and $\E\big[U_{k+}^{\PR}\big]$ is obtained by letting $K=k$ in
(\ref{eq:lemma-unfinished}). Equation (\ref{eq:lemma:E[U_k^PR]}) for
$k \in \K_2^-$ now follows from $\E[U_k^{\PR}] =
\E\big[U_{k+}^{\PR}\big] - \E\big[U_{(k-1)+}^{\PR}\big]$.
\end{proof}

\begin{theorem}
The mean waiting time $\E\big[W_k^{\PR}\big]$ ($k \in \K$) of
class $k$ customers in the stationary preemptive-resume priority queue
satisfying Assumptions~\ref{assumption-1} to \ref{assumption-3}
is given by
\begin{align*}
\E\left[W_k^{\PR}\right]
&=
\frac{\ds\sum_{\ell \in \K_k^+}
\rho_{\ell} \E[\tilde{H}_{\ell}]}
{(1-\rho_k^+)(1-\rho_{k-1}^+)}
+ 
\frac{\ds\sum_{\ell \in \K_k^+} \rho_{\ell} (\rho_k^+ - \rho_{\ell})}
{2(1-\rho_k^+)(1-\rho_{k-1}^+)}
+
\frac{\rho_{k-1}^+}{1-\rho_{k-1}^+}
\\
&\qquad {}
+
\frac{\ds\sum_{\ell \in \K_k^+}
\rho_{\ell}
\left[
\pi_{\ell}^{({\on})} \rho_{\ell} \left( 1+\E[\tilde{C}_{\ell}] \right)
+
\E[H_{\ell}]
\left(
\E[\tilde{\Lambda}_{\ell}] - \pi_{\ell}^{({\on})} 
\frac{\E[C_{\ell} \Lambda_{\ell}]}{\E[C_{\ell}]}
\right)
\right]
}
{(1-\rho_k^+)(1-\rho_{k-1}^+)}
\\
& \qquad\quad {}
+
\frac{\E[H_k]}{1-\rho_{k-1}^+} \left(
\E[\tilde{\Lambda}_k] - \pi_k^{({\on})}
\frac{\E[C_k \Lambda_k ]}{\E[C_k]} \right)
\\
&\qquad\qquad {}
+
\pi_k^{({\on})}
\left[
\left( 1 + \frac{\rho_k}{1-\rho_{k-1}^+} \right)
\left( 1+\E[\tilde{C}_k] \right)
-
\frac{\E[\Lambda_k(\tilde{C}_k)]}{\lambda_k}
\right],
\quad
k \in \K,
\end{align*}
where $\rho_0^+=0$.
\end{theorem}

\begin{proof}
The theorem for $k \in \K_2^-$ immediately follows from
(\ref{eq:U_k-W_k-PR}) and Lemma \ref{lemma:E[U_k^PR]}.  Recall that
$\E\left[W_1^{\PR}\right]$ is identical to $\E[W_1]$ in
(\ref{K=1}). Because $\rho_0^+=0$, we can confirm that the theorem
holds for $k=1$, too.
\end{proof}

Let $D_k^{\PR}$ ($k \in \K$) denote the system delay (i.e., the
sojourn time in the system) of a randomly chosen customer in class $k$
in the stationary preemptive-resume priority queue.

\begin{corollary}
The mean system delay $\E\big[D_k^{\PR}\big]$ ($k \in \K$)
of class $k$ in the stationary preemptive-resume priority queue 
satisfying Assumptions~\ref{assumption-1} to \ref{assumption-3}
is given by
\begin{equation}
\E\big[D_k^{\PR}\big] 
= 
\E\big[W_k^{\PR}\big] + \frac{\E[H_k]-\rho_{k-1}^+}{1-\rho_{k-1}^+},
\quad
k \in \K,
\label{eq:E[D_k^PR]}
\end{equation} 
where $\rho_0^+=0$.
\end{corollary}

\begin{proof}
The corollary follows from $\E\big[D_k^{\PR}\big] =
\E\big[W_k^{\PR}\big] + \E[H_k^{\PR}]$ and Lemma
\ref{lemma:mean_completion}.
\end{proof}

\section{The nonpreemptive priority queues}
\label{sec-nonpreemptive}

In this section, we consider the mean waiting times in the
nonpreemptive priority queue. Let $W_k^{\NP}$ ($k \in \K$) denote the
waiting time of a randomly chosen customer of class $k$ in steady
state.  Note that $W_K^{\NP} = W_K^{\PR}$ because the service of a class
$K$ customer can start only when there are no customers of classes 1
to $K-1$ both in the preemptive-resume and nonpreemptive priority
queues. It then follows that
\begin{equation}
\E\big[W_K^{\NP}\big] = E\big[W_K^{\PR}\big],
\label{eq:W_K^NP=W_K^PR}
\end{equation}
where $E\big[W_K^{\PR}\big]$ is given by (\ref{eq:E[D_k^PR]}).  We
thus consider the mean waiting time in class $k$ ($k \in \K_{K-1}^+$)
below.

Let $U_k^{\NP}$ ($k \in \K$) denote the amount of unfinished
work of class $k$ in steady state.

\begin{lemma}\label{lemma:E[U_k^NP]}
The mean amount $U_k^{\NP}$ ($k \in \K_{K-1}^+$) of unfinished work of
class $k$ is given by
\[
\E[U_k^{\NP}]
=
\E[U_k^{\PR}] 
- 
\frac{\rho_{k-1}^+}{1-\rho_{k-1}^+} \cdot \rho_k \E[\tilde{H}_k]
+
\frac{\rho_k \ds\sum_{\ell \in \K_{k+1}^-} \rho_{\ell} \E[\tilde{H}_{\ell}]}
{(1-\rho_{k-1}^+)(1-\rho_k^+)},
\quad
k \in \K_{K-1}^+,
\]
where $\rho_0^+=0$ and $\E[U_k^{\PR}]$ is given by
(\ref{eq:lemma:E[U_k^PR]}). 
\end{lemma}

\begin{proof}
To prove the lemma, we introduce an auxiliary model of a mixed
priority discipline, where customers of classes 1 to $k$ ($k \in
\K_{K-1}^+$) follow a preemptive-resume priority discipline, while
customers of classes $k+1$ to $K$ follow the nonpreemptive priority
discipline. For simplicity, we fix $k$ ($k \in \K_{K-1}^+$) and call
classes 1 to $k$ (resp.\ classes $k+1$ to $K$) the high-priority class
(resp.\ the low-priority class). As shown in Fig.\ \ref{figure:aux}
(a) and (b), the sum $U_{k+}^{\dag}$ of the total amounts of
unfinished work of the high-priority class and the remaining service
time of the low-priority class in the nonpreemptive priority queue and
in the auxiliary model are identical sample path-wise. In what
follows, we derive two different expressions for $\E[U_{k+}^{\dag}]$
by considering the nonpreemptive priority queue and the auxiliary
model.

We first consider the nonpreemptive priority queue. It is easy to see that 
(cf.\ Fig.\ \ref{figure:aux} (a))
\begin{equation}
\E\big[U_{k+}^{\dag}\big] 
= 
\sum_{\ell \in \K_k^+} \E\big[U_{\ell}^{\NP}\big] 
+ 
\sum_{\ell \in \K_{k+1}^-} 
\rho_\ell \big( 1 + \E\big[\tilde{H}_{\ell}\big] \big), 
\quad 
k \in K_{K-1}^+.
\label{eq:E[U_k^dag]-NP}
\end{equation}

Next, we derive $\E\big[U_k^{\dag}\big]$ by considering the auxiliary
model. We observe from Fig.\ \ref{figure:aux} (b) and (c) that
$U_k^{\dag}$ is composed of three factors: (i) the total amount of
unfinished work in high-priority customers (white area), (ii) the sum
of the service unit being served and the remaining service time of a
low-priority customer when she is being served (black area), and (iii)
the remaining service time of a low-priority customer when the service
is being preempted (striped area). It is easy to see that the
contribution of factor (i) to $\E\big[U_k^{\dag}\big]$ is given by
$\E[U_1^{\PR}] + \E[U_2^{\PR}] + \cdots + \E[U_k^{\PR}]$ and the
contribution of factor (ii) to $\E\big[U_k^{\dag}\big]$ is given by $1
+ \E[\tilde{H}_{\ell}]$ ($\ell \in \K_{k+1}^-$) with probability
$\rho_{\ell}$. We thus consider the contribution of factor (iii)
(i.e., striped area in Fig.\ \ref{figure:aux} (b)) to
$\E\big[U_k^{\dag}\big]$ below.

\begin{figure}
\centering
\subfigure[Nonpreemptive priority (original)]{%
\begin{picture}(150,30)(0,0)
\put(0,0){\vector(1,0){145}}
\put(150,0){\makebox(0,0){time}}
\put(5,0){\line(0,1){12}}
\multiput(5,12)(2,-2){6}{\line(1,0){2}}
\multiput(7,10)(2,-2){5}{\line(0,1){2}}
\put(17,2){\line(0,1){6}}
\multiput(17,8)(2,-2){4}{\line(1,0){2}}
\multiput(19,6)(2,-2){4}{\line(0,1){2}}
\multiput(35,0)(0.1,0){20}{\line(0,1){18}}
\multiput(37,0)(0.1,0){20}{\line(0,1){16}}
\multiput(39,0)(0.1,0){20}{\line(0,1){14}}
\multiput(41,0)(0.1,0){20}{\line(0,1){12}}
\multiput(43,0)(0.1,0){20}{\line(0,1){10}}
\multiput(45,0)(0.1,0){20}{\line(0,1){8}}
\multiput(47,0)(0.1,0){20}{\line(0,1){6}}
\multiput(49,0)(0.1,0){20}{\line(0,1){4}}
\multiput(51,0)(0.1,0){20}{\line(0,1){2}}
\multiput(35,18)(2,-2){9}{\line(1,0){2}}
\multiput(37,16)(2,-2){9}{\line(0,1){2}}
\put(41,14){\line(0,1){16}}
\multiput(41,30)(2,-2){11}{\line(1,0){2}}
\multiput(43,28)(2,-2){10}{\line(0,1){2}}
\put(63,10){\line(0,1){6}}
\multiput(63,16)(2,4){2}{\line(1,0){2}}
\multiput(65,16)(2,4){2}{\line(0,1){4}}
\multiput(67,24)(2,-2){9}{\line(1,0){2}}
\multiput(69,22)(2,-2){8}{\line(0,1){2}}
\put(85,8){\line(0,1){6}}
\multiput(85,14)(2,-2){6}{\line(1,0){2}}
\multiput(87,12)(2,-2){5}{\line(0,1){2}}
\put(97,4){\line(0,1){14}}
\multiput(97,18)(2,-2){8}{\line(1,0){2}}
\multiput(99,16)(2,-2){7}{\line(0,1){2}}
\put(113,4){\line(0,1){8}}
\multiput(113,12)(2,-2){6}{\line(1,0){2}}
\multiput(115,10)(2,-2){5}{\line(0,1){2}}
\put(125,2){\line(0,1){4}}
\multiput(125,6)(2,-2){3}{\line(1,0){2}}
\multiput(127,4)(2,-2){2}{\line(0,1){2}}
\multiput(131,0)(0.1,0){20}{\line(0,1){12}}
\multiput(133,0)(0.1,0){20}{\line(0,1){10}}
\multiput(135,0)(0.1,0){20}{\line(0,1){8}}
\multiput(131,12)(2,-2){3}{\line(1,0){2}}
\multiput(133,10)(2,-2){3}{\line(0,1){2}}
\put(140,5){\makebox(0,0){...}}
\end{picture}}
%%%%%%%%%%%%%%%%%%%%%%%%%%%%%%%%%%
\subfigure[Auxiliary model (PR+NR)]{%
\begin{picture}(150,30)(0,0)
\put(0,0){\vector(1,0){145}}
\put(150,0){\makebox(0,0){time}}
\put(5,0){\line(0,1){12}}
\multiput(5,12)(2,-2){6}{\line(1,0){2}}
\multiput(7,10)(2,-2){5}{\line(0,1){2}}
\put(17,2){\line(0,1){6}}
\multiput(17,8)(2,-2){4}{\line(1,0){2}}
\multiput(19,6)(2,-2){4}{\line(0,1){2}}
\multiput(35,0)(0.1,0){20}{\line(0,1){18}}
\multiput(37,0)(0.1,0){20}{\line(0,1){16}}
\multiput(39,0)(0.1,0){20}{\line(0,1){14}}
\multiput(35,18)(2,-2){3}{\line(1,0){2}}
\multiput(37,16)(2,-2){2}{\line(0,1){2}}
\put(41,0){\line(0,1){30}}
\multiput(41,30)(2,-2){11}{\line(1,0){2}}
\multiput(43,28)(2,-2){10}{\line(0,1){2}}
\put(41,12){\line(1,0){18}}
\multiput(41,0.5)(0,0.5){23}{\line(1,0){18}}
\multiput(59,0)(0.1,0){20}{\line(0,1){12}}
\multiput(61,0)(0.1,0){20}{\line(0,1){10}}
\put(63,0){\line(0,1){16}}
\multiput(63,16)(2,4){2}{\line(1,0){2}}
\multiput(65,16)(2,4){2}{\line(0,1){4}}
\multiput(67,24)(2,-2){9}{\line(1,0){2}}
\multiput(69,22)(2,-2){8}{\line(0,1){2}}
\put(63,8){\line(1,0){20}}
\multiput(63,0.5)(0,0.5){16}{\line(1,0){20}}
\multiput(83,0)(0.1,0){20}{\line(0,1){8}}
\put(85,0){\line(0,1){14}}
\multiput(85,14)(2,-2){6}{\line(1,0){2}}
\multiput(87,12)(2,-2){5}{\line(0,1){2}}
\put(85,6){\line(1,0){8}}
\multiput(85,0.5)(0,0.5){12}{\line(1,0){8}}
\multiput(93,0)(0.1,0){20}{\line(0,1){6}}
\multiput(95,0)(0.1,0){20}{\line(0,1){4}}
\put(97,0){\line(0,1){18}}
\multiput(97,18)(2,-2){8}{\line(1,0){2}}
\multiput(99,16)(2,-2){7}{\line(0,1){2}}
\put(97,2){\line(1,0){26}}
\multiput(97,0.5)(0,0.5){4}{\line(1,0){26}}
\multiput(123,0)(0.1,0){20}{\line(0,1){2}}
\put(113,4){\line(0,1){8}}
\multiput(113,12)(2,-2){6}{\line(1,0){2}}
\multiput(115,10)(2,-2){6}{\line(0,1){2}}
\put(125,2){\line(0,1){4}}
\multiput(125,6)(2,-2){3}{\line(1,0){2}}
\multiput(127,4)(2,-2){2}{\line(0,1){2}}
\multiput(131,0)(0.1,0){20}{\line(0,1){12}}
\multiput(133,0)(0.1,0){20}{\line(0,1){10}}
\multiput(135,0)(0.1,0){20}{\line(0,1){8}}
\multiput(131,12)(2,-2){3}{\line(1,0){2}}
\multiput(133,10)(2,-2){3}{\line(0,1){2}}
\put(140,5){\makebox(0,0){...}}
\end{picture}}
%%%%%%%%%%%%%%%%%%%%%%%%%%%%%%
\subfigure[Preemptive-resume Priority (high-priority class only)]{%
\begin{picture}(150,20)(0,0)
\put(0,0){\vector(1,0){145}}
\put(150,0){\makebox(0,0){time}}
\put(5,0){\line(0,1){12}}
\multiput(5,12)(2,-2){6}{\line(1,0){2}}
\multiput(7,10)(2,-2){5}{\line(0,1){2}}
\put(17,2){\line(0,1){6}}
\multiput(17,8)(2,-2){4}{\line(1,0){2}}
\multiput(19,6)(2,-2){4}{\line(0,1){2}}
\put(41,0){\line(0,1){18}}
\multiput(41,18)(2,-2){10}{\line(1,0){2}}
\multiput(43,16)(2,-2){9}{\line(0,1){2}}
\put(63,0){\line(0,1){8}}
\multiput(63,8)(2,4){2}{\line(1,0){2}}
\multiput(65,8)(2,4){2}{\line(0,1){4}}
\multiput(67,16)(2,-2){9}{\line(1,0){2}}
\multiput(69,14)(2,-2){8}{\line(0,1){2}}
\put(85,0){\line(0,1){8}}
\multiput(85,8)(2,-2){5}{\line(1,0){2}}
\multiput(87,6)(2,-2){4}{\line(0,1){2}}
\put(97,0){\line(0,1){16}}
\multiput(97,16)(2,-2){8}{\line(1,0){2}}
\multiput(99,14)(2,-2){7}{\line(0,1){2}}
\put(113,2){\line(0,1){8}}
\multiput(113,10)(2,-2){5}{\line(1,0){2}}
\multiput(115,8)(2,-2){5}{\line(0,1){2}}
\put(125,0){\line(0,1){6}}
\multiput(125,6)(2,-2){3}{\line(1,0){2}}
\multiput(127,4)(2,-2){3}{\line(0,1){2}}
\put(140,5){\makebox(0,0){...}}
\end{picture}}
\caption{Total amounts of unfinished work of classes 1 to $k$.  The
  black areas indicate the remaining service times of lower-priority
  class customers being served and the striped areas indicate
  remaining service times of waiting lower-priority class customers.}
\label{figure:aux}
\end{figure}
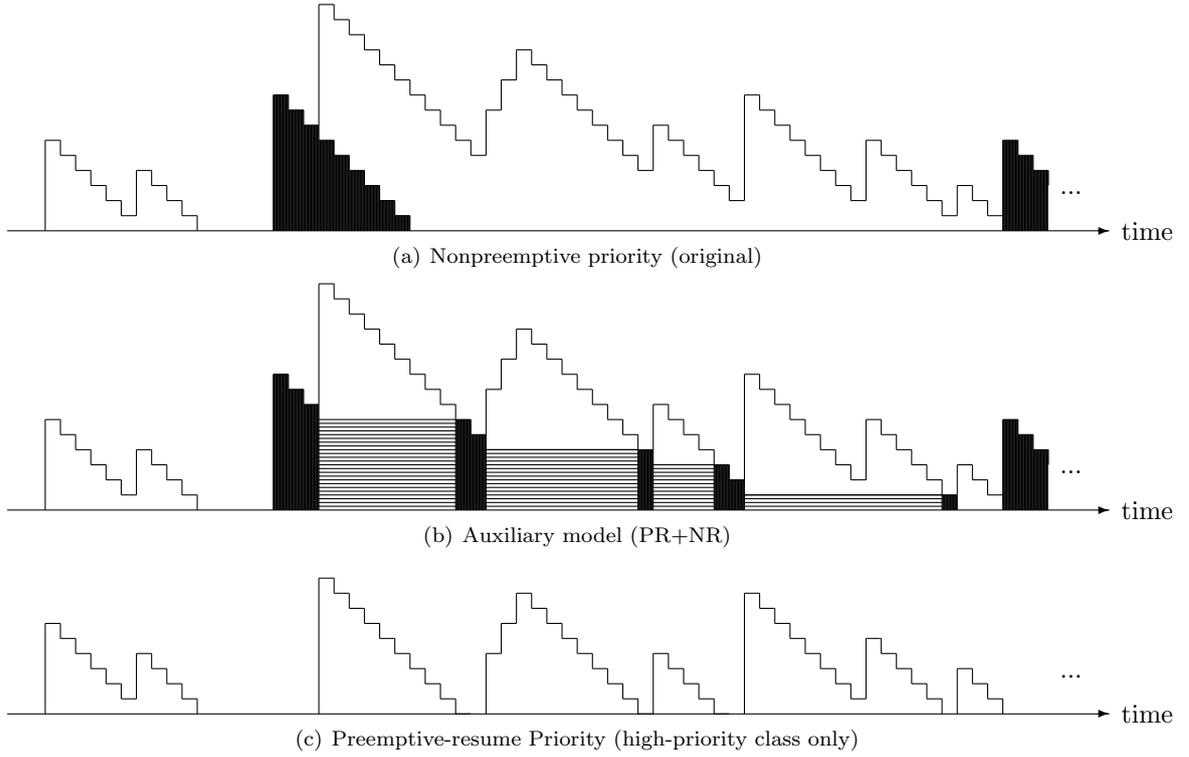

We observe that busy cycles of high-priority customers start only when
the server is idle or a low priority customer is being served.
Therefore, owing to the discrete-time version of conditional PASTA, a
busy cycle starts immediately after an idle period with probability
$(1-\rho)/(1-\rho_k^+)$, and if this is the case, there are no
contribution of factor (iii) to $\E\big[U_k^{\dag}\big]$, and
otherwise, the contribution of factor (iii) to
$\E\big[U_k^{\dag}\big]$ is given by $\E[\tilde{H}_{\ell}]$ ($\ell \in
\K_{k+1}^-$) with probability $\rho_{\ell}/(1-\rho_k^+)$. Based on
these observations, we obtain
\begin{equation}
\E\big[U_k^{\dag}\big]
=
\sum_{\ell \in \K_k^+} \E[U_{\ell}^{\PR}]
+
\sum_{\ell \in \K_{k+1}^-} \rho_{\ell}
\big( 1 + \E[\tilde{H}_{\ell}] \big)
+
\rho_k^+ 
\sum_{\ell \in \K_{k+1}^-} \frac{\rho_{\ell}}{1-\rho_k^+}
\E[\tilde{H}_{\ell}],
\quad
k \in \K_{K-1}^+.
\label{eq:E[U_k^dag]-aux}
\end{equation}
It then follows from (\ref{eq:E[U_k^dag]-NP}) and
(\ref{eq:E[U_k^dag]-aux}) that
\[
\sum_{\ell \in K_k^+} \E[U_{\ell}^{\NP}]
=
\sum_{\ell \in K_k^+} \E[U_{\ell}^{\PR}]
+
\frac{\rho_k^+}{1-\rho_k^+} 
\sum_{\ell \in \K_{k+1}^-} \rho_{\ell} \E[\tilde{H}_{\ell}],
\]
and therefore, 
\begin{align*}
\E[U_k^{\NP}]
&=
\sum_{\ell \in K_k^+} \E[U_{\ell}^{\NP}]
-
\sum_{\ell \in K_{k-1}^+} \E[U_{\ell}^{\NP}]
\\
&=
\E[U_k^{\PR}] 
+
\bigg(
\frac{\rho_k^+}{1-\rho_k^+} 
-
\frac{\rho_{k-1}^+}{1-\rho_{k-1}^+} 
\bigg)
\sum_{\ell \in \K_{k+1}^-} \rho_{\ell} \E[\tilde{H}_{\ell}]
- 
\frac{\rho_{k-1}^+}{1-\rho_{k-1}^+} \cdot \rho_k \E[\tilde{H}_k],
\end{align*}
from which the lemma follows.
\end{proof}

\begin{theorem}\label{theorem:W_k^NP}
The mean waiting time $\E\big[W_k^{\NP}\big]$ ($k \in \K$) of class
$k$ customers in the stationary nonpreemptive priority queue 
satisfying Assumptions~\ref{assumption-1} to \ref{assumption-3}
is given by
\[
\E\big[W_k^{\NP}\big]
=
\E \big[W_k^{\PR}\big]
+
\frac{\ds\sum_{\ell \in \K_{k+1}^-}
\rho_{\ell} \E\big[\tilde{H}_{\ell}\big]}
{(1-\rho_k^+)(1-\rho_{k-1}^+)},
\quad
k \in \K,
\]
where $\rho_0^+=0$ and $\K_{K+1}^- = \emptyset$.
\end{theorem}

\begin{remark}
It is interesting to observe that the difference between
$\E\big[W_k^{\NP}\big]$ and $\E \big[W_k^{\PR}\big]$ depends only on
the traffic intensity $\rho_j$ ($j \in \K$) and the mean remaining
service time $\E\big[\tilde{H}_{\ell}\big]$ ($\ell \in \K_{k+1}^-$)
and it is independent of the detailed structure of the arrival
processes.
\end{remark}

\begin{proof}
The theorem holds for $k=K$ because (\ref{eq:W_K^NP=W_K^PR}) and 
$\K_{K+1}^- = \emptyset$. Let $Q_k^{\NP}$ ($k \in \K_{K-1}^+$) 
denote the number of waiting customers in
class $k$ in the stationary nonpreemptive priority queue.  We then
have
\begin{align}
\E[U_k^{\NP}] 
&=
\E[H_k] \E\big[Q_k^{\NP}\big] + \rho_k \big(1 + \E\big[\tilde{H}_k\big]\big)
\nonumber
\\
&= 
\rho_k \E\big[W_k^{\NP}\big]
+ 
\rho_k \big(1 + \E\big[\tilde{H}_k\big]\big),
\quad
k \in \K_{K-1}^+,
\label{eq:U_k-W_k-NP}
\end{align}
where we use the Little's law in the second equality.  Therefore, it
follows from (\ref{eq:U_k-W_k-PR}), Lemma \ref{lemma:E[U_k^NP]}, and
(\ref{eq:U_k-W_k-NP}) that for $k \in \K_{K-1}^+$,
\begin{align*}
\lefteqn{%
\rho_k \E\big[W_k^{\NP}\big]
+ 
\rho_k \big(1 + \E\big[\tilde{H}_k\big]\big)
}\qquad\qquad
\\
& =
\rho_k \E\left[W_k^{\PR}\right]
+
\rho_k \bigg( 1 + \frac{\E[\tilde{H}_k]}{1-\rho_{k-1}^+} \bigg)
- 
\frac{\rho_{k-1}^+}{1-\rho_{k-1}^+} \cdot \rho_k \E[\tilde{H}_k]
+
\frac{\rho_k \ds\sum_{\ell \in \K_{k+1}^-} \rho_{\ell} \E[\tilde{H}_{\ell}]}
{(1-\rho_{k-1}^+)(1-\rho_k^+)}
\\
&=
\rho_k \E\left[W_k^{\PR}\right]
+
\rho_k \big( 1 + \E[\tilde{H}_k] \big)
+
\frac{\rho_k \ds\sum_{\ell \in \K_{k+1}^-} \rho_{\ell} \E[\tilde{H}_{\ell}]}
{(1-\rho_{k-1}^+)(1-\rho_k^+)},
\end{align*}
from which the theorem for $k \in \K_{K-1}^+$ follows.
\end{proof}

\section{Special cases}\label{section:special}

The queueing model in this paper is capable of representing various
arrival processes by setting $p_k$, $\bm{\alpha}_k(z)$, $\bm{T}_k(z)$
($k \in \K$) in (\ref{A_k(z)}) appropriately.  In this section, we
provide the results in three important special cases: the queue with
constant service times of one slot, the queue with i.i.d.\ arrivals,
and the queue with i.i.d.\ arrivals during active periods.

\subsection{Queues with constant service times of one slot}

Suppose $H_k=1$ ($k \in K$) with probability one. In this case, 
any service discipline is nonpreemptive and 
\[
\E[H_k] = 1, \qquad \E[\tilde{H}_k] = 0,
\quad
k \in \K.
\]
The conservation law is then given by
\begin{align*}
\sum_{k \in \K} \rho_k \E[W_k] 
&=
\frac{1}{2(1-\rho)} \sum_{k \in \K} \rho_k(\rho -\rho_k)  
+ \frac{1}{1-\rho} \sum_{k \in \K} 
\rho_k  \left( 
\E[\tilde{\Lambda}_k] 
- 
\pi_k^{[\on]} \frac{\E[C_k \Lambda_k ]}{\E[C_k]}
\right)
\nonumber
\\
& 
\qquad {} + 
\sum_{k \in \K} \pi_k^{[\on]} \rho_k 
\left( 1 + \frac{\rho_k}{1-\rho} \right)
\left(1 + \E[\tilde{C}_k] \right)
-
\sum_{k \in \K} \pi_k^{[\on]} \E[\Lambda_k(\tilde{C}_k)].
\end{align*}
Furthermore, the mean waiting times $\E[W_k^{\PR}]$ and $\E[W_k^{\NP}]$
($k \in \K$) in the preemptive-resume and nonpreemptive priority queues
are identical and given by
\begin{align*}
\E\left[W_k^{\PR}\right] 
=
\E\left[W_k^{\NP}\right] 
&=
\frac{\ds\sum_{\ell \in \K_k^+} 
\rho_{\ell} 
\left[ 
\pi_{\ell}^{({\on})} \rho_{\ell} \left( 1+\E[\tilde{C}_{\ell}] \right)
+
\E[\tilde{\Lambda}_{\ell}] - \pi_{\ell}^{({\on})} 
\frac{\E[C_{\ell} \Lambda_{\ell} ]}{\E[C_{\ell}]}
\right]
}
{(1-\rho_k^+)(1-\rho_{k-1}^+)} 
\\
& \qquad {} 
+ \frac{\ds\sum_{\ell \in \K_k^+} \rho_{\ell} (\rho_k^+ - \rho_{\ell})}
{2(1-\rho_k^+)(1-\rho_{k-1}^+)} 
+ 
\frac{\rho_{k-1}^+ +
\E[\tilde{\Lambda}_k] - \pi_k^{({\on})}
\ds\frac{\E[C_k \Lambda_k ]}{\E[C_k]}}
{1-\rho_k^+}
\\
& \qquad\qquad {} 
+
\pi_k^{({\on})}
\left[
\left( 1 + \frac{\rho_k}{1-\rho_{k-1}^+} \right)
\left( 1+\E[\tilde{C}_k] \right) 
-
\frac{\E[\Lambda_k(\tilde{C}_k)]}{\lambda_k}
\right] .
\end{align*}

\subsection{Queues with i.i.d.\ arrivals}

Suppose $\{A_{k,n}\}_{n=1,2,\ldots}$ ($k \in \K$) is a sequence of
i.i.d.\ random variables with probability mass function $a_k(m)
=\Pr(A_{k,n}=m)$ ($m=0,1,\ldots$), i.e.,
\[
\bm{A}_k(z) = \left(\begin{array}{cc}
a_k(0) & E[z^{A_k}] -a_k(0)\\
a_k(0) & E[z^{A_k}] -a_k(0)
\end{array}\right),
\qquad 
k \in \K,
\]
where $A_k$ ($k \in \K$) denotes a generic random variable for
$A_{k,n}$. Note that $\lambda_k = E[A_k]$, $\rho_k = \lambda_k
E[H_k]$, and $\pi_k^{[\on]} = 1-a_k(0)$. Furthermore, it is easy to
see that
\begin{align*}
E[\tilde{\Lambda}_k] 
= 
E[\tilde{A}_k] + \frac{\lambda_k}{a_k(0)},
&\qquad
E[\Lambda_k(\tilde{C}_k)]
= 
\frac{\lambda_k}{a_k(0)},
\\
E[\tilde{C}_k]
=
\frac{1-a_k(0)}{a_k(0)},
&\qquad
\frac{E[C_k \Lambda_k]}{E[C_k]} 
= 
\frac{\lambda_k (2-a_k(0))}{a_k(0) (1-a_k(0))} .
\end{align*}
The conservation law for nonpreemptive services is then given by
\[
\sum_{k \in \K} \rho_k \E[W_k] 
=
\frac{\rho}{1-\rho} \sum_{k \in \K} \rho_k \E[\tilde{H}_k] 
+ \frac{1}{2(1-\rho)} \sum_{k \in \K} \rho_k(\rho -\rho_k)  
+ \frac{1}{1-\rho} \sum_{k \in \K} 
\rho_k \E[H_k] \E[\tilde{A}_k],
\]
and the mean waiting time $E[W_k^{\PR}]$ 
($k \in \K$) in the preemptive-resume priority queue is given by
\[
\E\left[W_k^{\PR}\right] 
=
\frac{\ds\sum_{\ell \in \K_k^+} \rho_{\ell} \E[\tilde{H}_{\ell}]}
{(1-\rho_k^+)(1-\rho_{k-1}^+)}
+ 
\frac{\ds\sum_{\ell \in \K_k^+} \rho_{\ell} 
\left(
\rho_k^+ - \rho_{\ell} + 2\E[H_{\ell}] \E[\tilde{A}_{\ell}] 
\right)}
{2(1-\rho_k^+)(1-\rho_{k-1}^+)}
+
\frac{\rho_{k-1}^+ + \E[H_k] \E[\tilde{A}_k]}{1-\rho_{k-1}^+} .
\]
For $\E[W_k^{\NP}]$, see Theorem \ref{theorem:W_k^NP}.
These results agree with the well-known results
(cf. Section 6.6 in \cite{Taka93}).

\subsection{Queues with i.i.d.\ arrivals during active periods}

Suppose the number of class $k$ customers arriving at each slot during
active periods is independent of the state of the underlying Markov
chain.  Let $A_k^+$ ($k \in \K$) denote the conditional random
variable of $A_k$ given $A_k \geq 1$. We then have
\[
\bm{A}_k(z) = \left(\begin{array}{cc}
p_k & (1-p_k) a_k^+(z) \bm{\alpha}_k \\
(\bm{I}_k-\bm{T}_k) \bm{e}_k & a_k^+(z) \bm{T}_k
\end{array}\right),
\qquad 
k \in \K,
\]
where $a_k^+(z)= E[z^{A_k^+}] = E[z^{A_k} \mid A_k \geq 1]$. Note that
\begin{align*}
\pi_k^{[\on]} = \frac{\lambda_k}{E[A_k^+]} ,
&\qquad
E[\tilde{\Lambda}_k] = E[\tilde{A}_k^+] + E[A_k^+] E[\tilde{C}_k],
\\
E[\Lambda_k(\tilde{C}_k)] = E[A_k^+] E[\tilde{C}_k],
&\qquad
\frac{E[C_k \Lambda_k]}{E[C_k]} = 2 E[A_k^+] E[\tilde{C}_k] + E[A_k^+] .
\end{align*}
The conservation law for nonpreemptive services is then given by
\begin{align*}
\sum_{k \in \K} \rho_k \E[W_k] 
&=
\frac{\rho \ds\sum_{k \in \K} \rho_k \E[\tilde{H}_k]}{1-\rho} 
+ 
\frac{\ds\sum_{k \in \K} 
\rho_k \left[2 \E[H_k] \left( \E[\tilde{A}_k^+] - \E[A_k^+] \right)
+ 2-\rho+\rho_k \right]}{2(1-\rho)} 
\\
& \qquad {} 
+ 
\frac{\ds\sum_{k \in \K} 
\rho_k \left( \E[H_k] \E[A_k^+] - 1 + \rho -\rho_k\right)
\left(1 - \frac{\lambda_k}{\E[A_k^+]}\right) 
\left(1+\E[\tilde{C}_k] \right)}{1-\rho} ,
\end{align*}
and the mean waiting time $E[W_k^{\PR}]$ 
($k \in \K$) in the preemptive-resume priority queue is given by
\begin{align*}
\E\left[W_k^{\PR}\right] 
&=
\frac{\ds\sum_{\ell \in \K_k^+} \rho_{\ell} \E[\tilde{H}_{\ell}]}
{(1-\rho_k^+)(1-\rho_{k-1}^+)} 
+ 
\frac{\ds\sum_{\ell \in \K_k^+} 
\rho_{\ell} \left[2 \E[H_{\ell}] 
\left( \E[\tilde{A}_{\ell}^+] - \E[A_{\ell}^+] \right)
+ \rho_k^+ + \rho_{\ell} \right]}{2(1-\rho_k^+)(1-\rho_{k-1}^+)} 
\\
& \quad {} 
+ 
\frac{\ds\sum_{\ell \in \K_k^+} 
\rho_{\ell} \left(\E[H_{\ell}] \E[A_{\ell}^+] - \rho_{\ell} \right)
\left(1 - \frac{\lambda_{\ell}}{\E[A_{\ell}^+]}\right)
\left(1+\E[\tilde{C}_{\ell}] \right)}{(1-\rho_k^+)(1-\rho_{k-1}^+)} 
\\
& \qquad {} 
+ 
\frac{\E[H_k] \left( \E[\tilde{A}_k^+] - \E[A_k^+] \right)
+ 1 + \rho_k}{1-\rho_{k-1}^+}
\\
& \qquad\quad {} 
+ 
\frac{\left(
\E[H_k]\E[A_k^+] - 1 +\rho_{k-1}^+ - \rho_k
\right)
\left(1 - \ds\frac{\lambda_k}{\E[A_k^+]}\right)
\left(1+\E[\tilde{C}_k] \right)}
{1-\rho_{k-1}^+} .
\end{align*}
For $\E[W_k^{\NP}]$, see Theorem \ref{theorem:W_k^NP}.

\end{document}